\documentclass[10pt]{amsart}

\usepackage{amsmath,amsfonts,amssymb,amsthm}
\usepackage{bbm}

\newtheorem{theorem}{Theorem}[section]
\newtheorem{lemma}[theorem]{Lemma}
\newtheorem{proposition}[theorem]{Proposition}
\newtheorem{corollary}[theorem]{Corollary}

\theoremstyle{definition}

\newtheorem{example}[theorem]{Example}

\theoremstyle{remark}
\newtheorem{remark}[theorem]{Remark}

\numberwithin{equation}{section}

\newcommand{\nat}{\mathbb{N}}                                                   
\newcommand{\rzecz}{\mathbb{R}}                                                 
\newcommand{\zesp}{\mathbb{C}}                                                  
\newcommand{\iloskal}[2]{\left<{#1},{#2}\right>}                                
\DeclareMathOperator{\MojeIm}{Im}                                               
\newcommand{\norma}[2]{{\left\Arrowvert{#2}\right\Arrowvert}_{#1}}              
\newcommand{\pstwowar}[2]{\mathbb{P}\left({#1}\big|{#2}\right)}                 
\DeclareMathOperator{\Aut}{Aut}
\newcommand{\genbin}[2]{{#1 \brack #2}}
\newcommand{\genbinab}{{\alpha \brack \beta}}
\newcommand{\genbinba}{{\beta \brack \alpha}}

\newcommand{\jednorod}[1]{P_{#1}^{\textrm{h}}(V)}                                   
\newcommand{\chk}{\textrm{C}^\ast(H)^K}

\newcommand{\cgw}{\textrm{C}^\ast}
\newcommand{\dedt}{\frac{\partial}{\partial t}}
\newcommand{\ptplus}[2]{p_t^+(#1,#2)}

\newcommand{\symmc}{\textrm{Sym}(m,\zesp)}
\newcommand{\sph}[1]{\Phi_{\mathbf{#1}}}                                         

\newmuskip\pFqmuskip

\newcommand*\pFq[6][8]{%
	\begingroup 
	\pFqmuskip=#1mu\relax
	\mathcode`\,=\string"8000
	\begingroup\lccode`\~=`\,
	\lowercase{\endgroup\let~}\pFqcomma
	{}_{#2}F_{#3}{\left(\genfrac..{0pt}{}{#4}{#5};#6\right)}%
	\endgroup
}
\newcommand{\pFqcomma}{\mskip\pFqmuskip}

\begin{document}

\title[]{Multiplicity free actions, birth and death processes on partitions, and Biane's quantum Ornstein-Uhlenbeck semigroups}

\author[W.Matysiak]{Wojciech Matysiak}
\address{Wydzia{\l} Matematyki i Nauk Informacyjnych\\
Politechnika Warszawska\\
ul. Ko\-szy\-ko\-wa 75\\
00-662 Warsaw, Poland} \email{matysiak@mini.pw.edu.pl}

\author[M.\'{S}wieca]{Marcin \'{S}wieca}
\address{Wydzia{\l} Matematyki i Nauk Informacyjnych\\
Politechnika Warszawska\\
ul. Ko\-szy\-ko\-wa 75\\
00-662 Warsaw, Poland} \email{M.Swieca@mini.pw.edu.pl}

\thanks{Supported in part by the National Science Centre (Poland), Grant 2016/21/B/ST1/00005.}




\begin{abstract}
We introduce a family of multivariate continuous-time pure birth and pure death chains, with birth and death rates defined in terms of the generalized binomial coefficients for multiplicity free actions. 
The state spaces for some of the introduced processes are some sets of partitions (equivalently, Young diagrams). The chains turn out to be the classical Markov processes obtained by restricting Biane's quantum Ornstein-Uhlenbeck semigroups to commutative $\cgw$-algebras related to Gelfand pairs built on Heisenberg groups.
\end{abstract}

\maketitle

\section{Introduction}\label{s:intro}

The principal goal of this paper is to present a family of multivariate pure birth and pure death processes, with birth and death rates defined using  certain generalization of the binomial coefficients. The generalized binomial coefficients, introduced by Yan and studied in great detail by Benson and Ratcliff (see the references throughout this article), are inherently connected with an important class of group actions (the multiplicity free actions). The existence of a group action structure provides a natural framework for introducing some symmetries into the analyzed processes, what may be of interest from the point of view of potential applications.

In some cases, the state space for the birth and death processes we deal with is the set of partitions (of the 
length not greater than a fixed number). Since
the partitions can be represented as Young (Ferrer) diagrams, in those cases the processes we construct can be viewed as jump processes on the set of Young diagrams with 
number of rows not exceeding a fixed number, with jumps of two types: one either adds a box to a Young diagram (for the pure birth processes), or one removes a box from a diagram (for the pure death processes). In some cases, the dynamics on the set of Young diagrams is given by the generalized binomial coefficients for Jack polynomials (Schur polynomials in particular), studied by Lassalle (\cite{Lassalle1990}, \cite{Lassalle1998}), Okounkov and Olshanski \cite{OkounkovOlshanski1997}.  

The original motivation of this article was our interest in Biane's construction from \cite{biane1998} (see also \cite{biane1996}) of some  non-commutative analogues of the classical Ornstein-Uhlenbeck semigroups. They are certain semigroups of positive contractions on a (non-commutative) $\cgw$-algebra, which are related to the semigroup of a non-commutative Brownian motion (a non-commutative version of the heat semigroup) in a similar way as the classical Ornstein-Uhlenbeck semigroups are related to the semigroup of the classical Brownian motion. 
Let us mention that besides Biane's quantum Ornstein-Uhlenbeck semigroups, a number of others have been defined with the same name, see e.g. \cite{CiprianiFagnolaLindsay} or \cite{KoPark}.

Generally speaking, the semigroups of completely positive contractions may be of interest from the point of view of classical probability theory as they often give rise, via restrictions to commutative sub-$\cgw$-algebras, to interesting classical Markov processes. Some results from Biane's work \cite{biane1998} serve as representative examples of such a phenomenon (see also \cite{biane1996}, \cite{matysiakswieca1}, \cite{matysiakswieca2} and \cite{matysiakhyper})
; likewise, the birth and death processes appearing in Section \ref{s:examples} of this paper, in the connection with Theorem \ref{thm:connection}, can be regarded in the same fashion.

The commutative sub-$\cgw$-algebras Biane uses to restrict his Ornstein-Uhlenbeck semigroups to, are related to some Gelfand pairs $(K,H)$,  associated with Heisenberg groups $H$ ($K$ denotes a
closed subgroup of a unitary group). Specifically, Biane computes the relevant restrictions for the realization $H=\zesp^n\times\rzecz$ of the Heisenberg group and two choices of $K$, namely $K=U(n)$ and $K=\textrm{SO}(\rzecz,n)\times U(1)$ (for a restricted range of $n$). In the first case, the classical Markov processes turn out to be a pure birth process known as Yule process, and a pure death process, both on non-negative integers. In the second case, Biane obtained birth and death rates of a bivariate pure birth and a bivariate pure death process.

Our contribution in the present paper is twofold. First, we were able to find birth and death rates, as well as semigroups of transition probabilities, of the pure birth and pure death processes arising from the restriction of Biane's quantum Ornstein-Uhlenbeck semigroups to the commutative $\cgw$-algebra generated by any Gelfand pair $(K,H)$ (for any admissible subgroups $K$), thus completing and extending   Biane's results from \cite{biane1998}. By linking the quantum Ornstein-Uhlenbeck semigroups with the  well-developed theory of Gelfand pairs associated with Heisenberg groups, we identified and expressed the birth and death intensities in terms of the generalized binomial coefficients for multiplicity free actions; this, in turn, enabled us to explicitly compute them for a wide range of examples. 

Second, we provide an elementary computation (with no appeal to Heisenberg groups, Gelfand pairs or quantum Ornstein-Uhlenbeck semigroups) of the same semigroups of transition probabilities that we found when extending the results from \cite{biane1998}. Here our point of departure are the intensities, given in terms of the generalized binomial coefficients. The relevant part of the present article precedes the part dealing with the more advanced material, thus making it possible to be read for its own sake.

The general layout of the paper is as follows. We begin in Section \ref{s:multi} with an introduction of the generalized binomial coefficients for multiplicity free actions. In Section \ref{s:birth-death-Kaction} we use the coefficients to define intensities of some multivariate pure birth and pure death processes; in the same section we compute the relevant semigroups of transition probabilities in an elementary way. In Section \ref{s:examples} some explicit examples illustrating the theory from the previous section are presented. In Section \ref{s:heis} we refer to a more advanced material and report Biane's results on quantum Ornstein-Uhlenbeck semigroups. We also prove there that restricting the quantum semigroups to commutative sub-$\cgw$-algebras built from Gelfand pairs, yields exactly the same classical Markov processes as we defined in Section \ref{s:birth-death-Kaction}. All auxiliary results along with their proofs are collected in the last section.

\section{Multiplicity free actions, invariant polynomials and generalized binomial coefficients}\label{s:multi}

Consider a finite dimensional Hermitian inner product space $\left(V,\iloskal{\cdot}{\cdot}\right)$ with complex dimension $\dim_\zesp V=n$, and a closed subgroup $K$ of the group $U(V)$ of unitary operators on $V$. By $k.z$ we will be denoting the natural action of $k\in K$ on a vector $z\in V$. The associated representation of $K$ on the space $\zesp[V]$ of holomorphic polynomials on $V$ is given by
\begin{equation}\label{polyrep}
(k.p)(z)=p\left(k^{-1}.z\right)
\end{equation}
for $k\in K$, $p\in\zesp[V]$ and $z\in V$, and it generates a decomposition of $\zesp[V]$ into an algebraic direct sum 
\begin{equation}\label{decompos}
\zesp[V]=\bigoplus_{\lambda\in\Lambda} P_\lambda
\end{equation}
of finite-dimensional $K$-irreducible subspaces $P_\lambda$ ($\Lambda$ is a countably infinite index set parametrizing the decomposition). 

Clearly, the spaces $\jednorod{m}$ of homogeneous polynomials of degree $m$ on $V$ are $U(V)$-invariant, hence each $P_\lambda$ is a subspace of some $\jednorod{m}$. More precisely, if $|\lambda|$ stands for the degree of homogeneity of the polynomials from $P_\lambda$, then $P_\lambda\subset\jednorod{|\lambda|}$. Let us also write $0\in \Lambda$ for the index with $P_0=\jednorod{0}\cong\zesp$, and $d_\lambda$ for the dimension of $P_\lambda$. 

In a series of papers (\cite{BJR92}, \cite{BensonRatcliff1996}, \cite{BensonRatcliff98}, \cite{BensonRatcliff2000}, \cite{BensonRatcliff2004}, \cite{BensonRatcliff2009}, \cite{BensonHoweRatcliff}), Benson and Ratcliff, along with some coauthors, studied a special class of $K$-actions (named \emph{multiplicity free}), originally for their connection with analysis on the Heisenberg group. The action of $K$ on $V$ is said to be multiplicity free when the representations of $K$ on $P_\alpha$ and $P_\beta$ (see \eqref{decompos}) are inequivalent for $\alpha\ne\beta$. That is, no irreducible representation of $K$ occurs more than once in $\zesp[V]$. One can show that then the decomposition \eqref{decompos} is canonical. 

Multiplicity free actions are well-studied, because, \emph{inter alia},  of their importance in representation theory and invariant theory (see \cite{Howe1995}); in particular, they have been completely classified (see \cite{kac1980}, \cite{BensonRatcliff1996}, and \cite{Leahy}).

Throughout the paper, we will assume that the action of $K$ on $V$ is multiplicity free. Now we will summarize the elements of the theory developed by Benson and Ratcliff (and their coauthors), which will be relevant to our considerations.

\subsection{Some polynomials associated with the $K$-action}\label{ss:polynomials}

We will be interested in $K$-invariant polynomials. Observe, however, that there are no non-constant $K$-invariants in $\zesp[V]$. Indeed, if there existed such $p\in\zesp[V]$, the space $\zesp p\cong\zesp$ would appear in the decomposition \eqref{decompos}, contrary to the assumption of multiplicity-freeness. Therefore we 
consider the larger ring $\zesp[V_\rzecz]$ of complex valued polynomial functions on the underlying real vector space $V_\rzecz$ of $V$, and accompanying algebra $\zesp[V_\rzecz]^K$ of $K$-invariant elements of $\zesp[V_\rzecz]$. The latter is non-trivial, with
\begin{equation}\label{gamma}
\gamma(z)=\iloskal{z}{z}
\end{equation}
being an example of element of $\zesp[V_\rzecz]^K$ for any $K$-action. In fact, we are able to describe all elements of $\zesp[V_\rzecz]^K$. To this end, we endow $\zesp[V]$ with the inner product $\iloskal{\cdot}{\cdot}_{1}$, where $\iloskal{\cdot}{\cdot}_{a}$ for any positive $a$ denotes the Fock inner product
\begin{equation}\label{fockinner}
\iloskal{f}{g}_{a}=\left(\frac{a}{\pi}\right)^n\int_V f(z)\overline{g(z)}\exp{[-a\gamma(z)]}\textrm{d}z
\end{equation}
($\textrm{d}z$ is Lebesgue measure on $V_\rzecz$). Next, for any $\lambda\in\Lambda$ and any orthonormal basis $\{v_1,\ldots,v_{d_\lambda}\}$ of $(P_\lambda,\iloskal{\cdot}{\cdot}_1)$, we define an $\rzecz^+$ valued polynomial from $\zesp[V_\rzecz]$
\begin{equation}\label{plambda}
p_\lambda(z)=\frac{1}{d_\lambda}\sum_{j=1}^{d_\lambda} v_j(z) \overline{v_j(z)}.
\end{equation}
As it was proved in \cite[Proposition 3.9]{BJR92}, this definition is independent of the basis chosen for $P_\lambda$ (which implies that $p_\lambda$ is $K$-invariant as the action of $K$ on $\{v_1,\ldots,v_{d_\lambda}\}$ yields another orthonormal basis of $P_\lambda$), and $\{p_\lambda\}_\lambda$ is itself a vector space basis for $\zesp[V_\rzecz]^K$. Clearly, $p_\lambda$ is homogeneous of degree $2|\lambda|$.

One can construct another canonical basis $\{q_\lambda\}_\lambda$ for $\zesp[V_\rzecz]^K$, applying the Gram--Schmidt orthogonalization process in the space $\left(\zesp[V_\rzecz]^K,\iloskal{\cdot}{\cdot}_1\right)$ to $\{p_\lambda\}_\lambda$. This of course requires an ordering of the index set $\Lambda$ and it was proved in \cite[Proposition 4.2]{BJR92} that any one in which $\alpha$ precedes $\beta$ whenever $|\alpha|<|\beta|$ (the indices $\lambda$ with the same value of $|\lambda|$ can be ordered arbitrarily), together with normalization $q_\lambda(0)=1$, produces the same  set $\{q_\lambda\}_\lambda$. 
Observe that $(-1)^{|\lambda|} p_\lambda$ is the homogeneous component of highest degree in $q_\lambda$, that is
\begin{equation}\label{homcomphighdeg}
q_\lambda=(-1)^{|\lambda|}p_\lambda + \textrm{ lower order terms.}
\end{equation}

\subsection{Generalized binomial coefficients associated with the $K$-action on $V$}\label{ss:genbin}

Thus we have defined two canonical bases, $\{p_\lambda\}$ and $\{q_\lambda\}$, for the vector space $\mathbb{C}[V_\rzecz]^K$. The first consists of some homogeneous polynomials, and the second of some orthogonal ones. The interplay between these two bases allows to introduce certain generalization of the binomial coefficients associated with the action of $K$ on $V$. Writing $q_\lambda$'s as linear combinations of $p_\lambda$'s
\begin{equation}\label{qinp}
q_\lambda(z)=\sum_{|\alpha|\le|\lambda|} (-1)^{|\alpha|}\genbin{\lambda}{\alpha} p_\alpha(z)
\end{equation}
we in fact define some real numbers $\genbin{\lambda}{\alpha}$ for any $\lambda, \alpha\in\Lambda$ with $|\lambda|\ge|\alpha|$ (additionally, we extend the definition to all of $\Lambda\times\Lambda$ by putting $\genbin{\lambda}{\alpha}=0$ for $|\alpha|>|\lambda|$). The 
numbers are called the \emph{generalized binomial coefficients for the multiplicity free action of $K$ on $V$}, which is justified by the fact that for $K=U(V)$ they boil down to the ordinary binomial coefficients (see Example \ref{ex:ordbincoeff}).

The generalized binomial coefficients were introduced by Yan in \cite{yan_unpub}; Benson and Ratcliff proved that they are rational numbers \cite{BensonRatcliff2000} 
(the factor $(-1)^{|\alpha|}$ in \eqref{qinp} guarantees they are non-negative), and 
studied their rich combinatorial structure \cite{BensonRatcliff98}. 
We will reveal many properties of the coefficients throughout this paper. For now, let us only observe that \eqref{homcomphighdeg} implies that if $|\alpha|=|\beta|$ then
\begin{equation}\label{zeroone}
\genbin{\alpha}{\beta}=
\begin{cases}
1,& \alpha=\beta,\\
0,& \alpha\ne\beta.
\end{cases}
\end{equation}

\begin{remark}
The $K$-invariant polynomials $(p_\lambda)_\lambda$ given by \eqref{plambda} essentially depend upon the decomposition \eqref{decompos} generated by the representation of $K$ on $\zesp[V]$ (and not upon the group $K$ itself). Therefore the generalized binomial coefficients defined by \eqref{qinp} also fundamentally depend on the decomposition \eqref{decompos}.
\end{remark}

\begin{remark}\label{rem:monoid}
Howe and Umeda (see \cite[p. 571]{HoweUmeda} or the end of Section 3 in \cite{BensonRatcliff2004}) proved that the set $\Lambda$ which parametrizes the decomposition of $\zesp[V]$ into $K$-irreducibles, is a free abelian monoid and that there exist 
an integer $r\ge 1$ and linearly independent weights $\eta_1,\dots,\eta_r$ such that 
$\Lambda=\nat\eta_1+\ldots+\nat\eta_r$. (The number $r$ is called the \emph{rank of the multiplicity free action} of $K$ on $V$.) 
Therefore, every $\lambda\in\Lambda$ can be identified with an element $(\lambda_1,\ldots,\lambda_r)$ of $\nat^r$; moreover, the identification can be made so that $|\lambda|=\lambda_1+\ldots+\lambda_r$. We are going to view $\Lambda$ in such a way for the remainder of this paper.
\end{remark}

(Throughout, $\nat$ denotes the non-negative integers, including zero.)

\begin{remark}\label{rem:conventions}
As we will rely on some formulas from the paper of Benson and Ratcliff \cite{BensonRatcliff98}, it is important to point out slight differences between some objects appearing in our paper and in \cite{BensonRatcliff98}. Benson and Ratcliff \cite{BensonRatcliff98} take $V\ni z\mapsto\frac{1}{2}\iloskal{z}{z}$ as the distinguished $K$-invariant element of $\zesp[V_\rzecz]$ (compare with \eqref{gamma}) and use the inner product $\iloskal{\cdot}{\cdot}_{1/2}$ (see \eqref{fockinner}) in their computations. It is not difficult to observe that the conventions they adopt lead to bases $\{\widetilde{p}_\lambda\}$ and $\{\widetilde{q}_\lambda\}$ for $\zesp[V_\rzecz]^K$, which are related to the ones we defined in Section \ref{ss:genbin} via
\[
\widetilde{p}_\lambda(z)=p_\lambda\left(\frac{z}{\sqrt{2}}\right),\ \widetilde{q}_\lambda(z)=q_\lambda\left(\frac{z}{\sqrt{2}}\right).
\]
This immediately implies that the generalized binomial coefficients from \eqref{qinp} are exactly the same as the ones defined by Benson and Ratcliff in \cite[(1.4)]{BensonRatcliff98}. Note that there are some typos in the analogous remarks placed on page 82 of \cite{BensonRatcliff98}.
\end{remark}

In what follows, when referring to a result from \cite{BensonRatcliff98}, we will in fact quote its equivalent version, in which we will use our conventions.

\section{Birth and death processes for the $K$-action}\label{s:birth-death-Kaction}

We shall first employ the theory described in Section \ref{s:multi} to define certain birth and death processes taking values in $\Lambda$. In the light of Remark \ref{rem:monoid}, we will view them as processes on $\nat^r$, where $r$ is the rank of $K$-action.

Consider the semigroup $(p^+_t)_{t\ge0}$ of transition probabilities of a birth process given implicitly by the transition rates $q^+(\alpha,\beta)=\partial p^+_t(\alpha,\beta)/\partial t|_{t=0}$ of the form
\begin{equation}\label{q+}
q^+(\alpha,\beta)=
\begin{cases}
\frac{d_\beta}{d_\alpha}\genbin{\beta}{\alpha}, & |\beta|=|\alpha|+1,\\
-n-|\alpha|, & \alpha=\beta,\\
0,& \textrm{otherwise,}
\end{cases}                     
\end{equation}
and the semigroup $(p_t^-)_{t\ge0}$ of transition probabilities of a death process defined analogously by the transition rates $q^-(\alpha,\beta)=\partial p^-_t(\alpha,\beta)/\partial t|_{t=0}$ of the form
\begin{equation}\label{q-}
q^-(\alpha,\beta)=
\begin{cases}
\genbinab, & |\alpha|=|\beta|+1,\\
-|\alpha|,& \alpha=\beta,\\
0,& \textrm{otherwise.}
\end{cases}
\end{equation}
From Corollary \ref{cor:sums_of_rates} and nonnegativity of the generalized binomial coefficients, it follows that the rates \eqref{q+} and \eqref{q-} are well-defined.

\begin{theorem}\label{thm:semibirth} The semigroups $(p_t^\pm)_t$ are given by the following formulas:
\begin{eqnarray}
p_t^+(\alpha,\beta) &=&
\begin{cases}
\genbin{\beta}{\alpha}\frac{d_\beta}{d_\alpha}\left(1-e^{-t}\right)^{|\beta|-|\alpha|}\left(e^{-t}\right)^{|\alpha|+n},&\textrm{ if}\ |\beta|\ge|\alpha|,\\
0,& \textrm{ otherwise,}
\end{cases}\label{ptplus}\\
p_t^-(\alpha,\beta) &=& 
\begin{cases}
\genbin{\alpha}{\beta}e^{-|\beta|t}\left(1-e^{-t}\right)^{|\alpha|-|\beta|},&\textrm{ if}\ |\alpha|\ge|\beta|,\\
0, & \textrm{otherwise.}
\end{cases}\label{ptminus}
\end{eqnarray}
\end{theorem}

\begin{proof}
Fix $\alpha\in\Lambda$. By \eqref{q+}, the forward Kolmogorov equation
\begin{equation*}
\dedt p_t^+(\alpha,\beta)=\sum_{\lambda\ne\beta} p_t^+(\alpha,\lambda) q^+(\lambda,\beta)+ p_t^+(\alpha,\beta)q^+(\beta,\beta)
\end{equation*}
takes the form
\begin{equation}\label{kolmogorov}
\dedt p_t^+(\alpha,\beta)=\sum_{|\lambda|=|\beta|-1} \genbin{\beta}{\lambda}\frac{d_\beta}{d_\lambda}\ptplus{\alpha}{\lambda}
-(n+|\beta|)\ptplus{\alpha}{\beta}.
\end{equation}
We will solve \eqref{kolmogorov}, considering three cases.

In the first one, we assume that $|\beta|<|\alpha|$ and proceed by induction on $|\beta|$. If $|\beta|=0$ then \eqref{kolmogorov} becomes
$\dedt \ptplus{\alpha}{\beta}=-n \ptplus{\alpha}{\beta}$, and together with the initial condition $p_0^+(\alpha,\beta)=0$, implies \begin{equation}\label{ptplus0}
\ptplus{\alpha}{\beta}=0\quad\forall t.
\end{equation} 
Now we assume \eqref{ptplus0} for $|\beta|=0,1,\ldots,k$ with $\nat\ni k<|\alpha|$, and take $\beta\in\Lambda$ such that $|\beta|=k+1<|\alpha|$. By the induction hypothesis, for $|\lambda|=|\beta|-1=k$,  $\ptplus{\alpha}{\lambda}$ vanishes for all $t$, so \eqref{kolmogorov} becomes $\dedt\ptplus{\alpha}{\beta}=-(n+|\beta|)\ptplus{\alpha}{\beta}$, and we arrive at \eqref{ptplus0} since the initial condition $p_0^+(\alpha,\beta)=0$ holds. This completes the induction, and the proof of \eqref{ptplus0} for $|\beta|<|\alpha|$. 

As the second case, we consider $|\beta|=|\alpha|$, which, by the first case, leads to 
\begin{equation*}
\left\{
\begin{aligned}
\dedt\ptplus{\alpha}{\beta} &=-(n+|\beta|)\ptplus{\alpha}{\beta},\\
p_0^+(\alpha,\beta) &=\delta_{\alpha,\beta},
\end{aligned}
\right.
\end{equation*}
where $\delta_{\alpha,\beta}$ stands for the Kronecker delta. This leads to $\ptplus{\alpha}{\beta}=\exp\left[-(n+|\beta|)t\right]$ if $\beta=\alpha$, and to $\ptplus{\alpha}{\beta}=0$ for all $t$ otherwise,
what, by \eqref{zeroone}, proves \eqref{ptplus} for $|\beta|=|\alpha|$.

Finally, we consider $|\beta|\ge|\alpha|$ and prove \eqref{ptplus} by induction on $|\beta|$, starting from $|\beta|=|\alpha|$. The basis was proved above. For the inductive step we pick a non-negative integer $k$, and assume that $\ptplus{\alpha}{\beta}$ has the claimed form for $|\beta|=|\alpha|,\ldots,|\alpha|+k$. Take $\beta$ with $|\beta|=|\alpha|+k+1$. Combining \eqref{kolmogorov} and the inductive hypothesis we obtain
\begin{equation*}
\dedt\ptplus{\alpha}{\beta}=-(n+|\beta|)\ptplus{\alpha}{\beta}+\exp\left[-(|\alpha|+n)t\right]\frac{d_\beta}{d_\alpha}
\sum_{|\lambda|=|\beta|-1}\genbin{\beta}{\lambda}\genbin{\lambda}{\alpha}(1-e^{-t})^{|\lambda|-|\alpha|}.
\end{equation*}
Benson and Ratcliff \cite[Corollary 3.4]{BensonRatcliff98} proved that if $\alpha,\beta\in\Lambda$, $|\alpha|\le|\beta|$, then for all $l\in\{|\alpha|,\ldots,|\beta|\}$
\begin{equation}\label{BRCorr34}
\sum_{|\lambda|=l}\genbin{\beta}{\lambda}\genbin{\lambda}{\alpha}=\frac{(|\beta|-|\alpha|)!}{(|\beta|-l)!(l-|\alpha|)!}\genbin{\beta}{\alpha}.
\end{equation}
Therefore
\begin{equation*}
\dedt\ptplus{\alpha}{\beta}=-(n+|\beta|)\ptplus{\alpha}{\beta}+\genbin{\beta}{\alpha}(|\beta|-|\alpha|)\frac{d_\beta}{d_\alpha}
e^{-(|\alpha|+n)t}(1-e^{-t})^{|\beta|-|\alpha|-1}.
\end{equation*}
Solving this equation with the initial condition $p_0^+(\alpha,\beta)=0$ confirms that $\ptplus{\alpha}{\beta}$ has the claimed form for $|\beta|=|\alpha|+k+1$, completing the induction and the proof of \eqref{ptplus}.

We omit the proof of \eqref{ptminus}, since it proceeds along the same lines, with the backward Kolmogorov equation as a point of departure.
\end{proof}
A number of examples will be presented in the next section to shed some light on the scope of the multidimensional birth and death processes arising from the multiplicity free actions. Here we offer only the example that largely motivates the laid out theory.

\begin{example}\label{ex:ordbincoeff} Let $V=(\zesp^n,\iloskal{\cdot}{\cdot}_{\zesp^n})$, where  $\iloskal{u}{v}_{\zesp^n}=\sum_{i=1}^n u_i \overline{v_i}$. Consider the usual action $k.z=kz$ ($k\in U(n), z\in\zesp^n$) of $K=U(V)=U(n)$ on $V$. It is well known that under this action, the ring $\zesp[V]$ decomposes as
	\begin{equation*}
	\zesp[V]=\bigoplus_{m\in\nat}P_m^{\textrm{h}}(\zesp^n)
	\end{equation*}
	(so here $\Lambda=\nat$ 
	and $r=1$). As the orthonormal basis for $P_m^{\textrm{h}}(\zesp^n)$ (with respect to $\iloskal{\cdot}{\cdot}_1$ - recall \eqref{fockinner}), take \begin{equation}\label{basis}
	\left\{\frac{z^\alpha}{\sqrt{\alpha!}}:\ \alpha=(\alpha_1,\ldots,\alpha_n)\in\nat^n, \sum_{i=1}^n\alpha_i=m\right\}
	\end{equation}
	(we are using the standard multi-index notation: $z^\alpha=z_1^{\alpha_1}\cdot\ldots\cdot z_n^{\alpha_n}$ for $z=(z_1,\ldots,z_n)\in\zesp^n$, and $\alpha!=\alpha_1!\cdot\ldots\cdot\alpha_n!$).
	As it is also well known that
	\[
	\dim P_m^{\textrm{h}}(\zesp^n)=\binom{m+n-1}{n-1},
	\]
	it follows easily from \eqref{plambda} that
	\begin{equation}\label{pmUn}
	p_m(z)=\binom{m+n-1}{n-1}^{-1}\sum_{|\alpha|=m}\frac{z^\alpha}{\sqrt{\alpha!}}\overline{\left(\frac{z^\alpha}{\sqrt{\alpha!}}\right)}
	=\frac{(n-1)!}{(m+n-1)!}\gamma^m(z)
	\end{equation}
	(recall the definition of $\gamma$ from \eqref{gamma}).
	The polynomials $(q_m)_m$ are obtained from $p_m$'s as an orthonormal sequence with respect to the Gaussian measure 
	$\exp(-\gamma(z))\textrm{d}z$. As $p_m$'s and the Gaussian kernel $\exp(-\gamma(z))$ are functions of $\sum|z_i|^2$, it is not difficult to observe (for details, see \cite[Proposition 6.2]{BJR92}) that this orthogonalization problem is equivalent to the classical one of orthogonalizing the monomials $(x^k)_{k=0,1,\ldots}$ in $L^2([0,\infty),x^{n-1}e^{-x})$. Since the latter yields the generalized Laguerre polynomials $(L_k^{(n-1)})_k$ of order $n-1$
	\begin{equation}\label{laguerre}
	L_k^{(n-1)}(x)=\frac{(n)_k}{k!} \pFq{1}{1}{-k}{n}{x}=\sum_{i=0}^k (-1)^i\binom{k+n-1}{k-i}\frac{x^i}{i!},\quad x\in\rzecz,
	\end{equation}
	we get the generalized Laguerre polynomials $L_m^{(n-1)}(\gamma(z))$, normalized so that $L_m^{(n-1)}(0)=1$:
	\begin{equation}\label{qmUn}
	q_m(z)=(n-1)!\sum_{j=0}^m \binom{m}{j}\frac{[-\gamma(z)]^j}{(j+n-1)!}.
	\end{equation} 
	Looking at \eqref{qinp}, \eqref{pmUn} and \eqref{qmUn}, we notice that in this case
	\begin{equation*}
	\genbin{m}{j}=\binom{m}{j},
	\end{equation*}
	so that the generalized binomial coefficients for the standard action of $K=U(n)$ on $V=\zesp^n$ are the ordinary binomial coefficients. Naturally, this fact justifies the terminology.
	
	Now we immediately see that for the usual action of $U(\zesp^n)$ on $\zesp^n$, Theorem \ref{thm:semibirth} boils down to the well-known facts that the semigroups $(p_t^+)_t$ of a pure birth process, and $(p_t^-)_t$ of a pure death process, on $\nat$, are
	\begin{equation}\label{pt+-Un}
	\begin{array}{r@{}l}
	p_t^+(k,l) &{}= \begin{cases}
	\binom{l+n-1}{k+n-1}\left(1-e^{-t}\right)^{l-k}\left(e^{-t}\right)^{k+n},&\textrm{ if}\ l\ge k,\\
	0,& \textrm{ otherwise,}
	\end{cases}\\
	p_t^-(k,l) &{}=\begin{cases}
	\binom{k}{l}e^{-lt}\left(1-e^{-t}\right)^{k-l},&\textrm{ if}\ k\ge l,\\
	0, & \textrm{otherwise,}
	\end{cases}
	\end{array}
	\end{equation}	
	if their birth rates (resp. death rates) equal $k+n$ (resp. $k$). 
	In particular, the pure birth process is a Yule process with immigration ($n$ represents the rate of immigration, the individual birth rate is $1$).
\end{example}
An interesting feature of all pure birth and pure death processes emerging from the multiplicity free actions is exhibited in the following proposition.

\begin{proposition}\label{prop:sum}
If $(X_t)_t$ stands for the pure birth (resp. pure death) process defined by \eqref{q+} (resp. \eqref{q-}), then $(|X_t|)_t$ is the pure birth (resp. pure death) process with the transition semigroup $(p_t^+)_t$ (resp. $(p_t^-)_t$) from \eqref{pt+-Un}.
\end{proposition}

\begin{proof}
We will prove that $(|X_t|)_t$ is a Markov process and that its transition probabilities coincide with the relevant part of \eqref{pt+-Un}.

By Theorem \ref{thm:semibirth}, for $0\le s<t$, $k\in \nat$ and any $\alpha\in\Lambda$,
\begin{multline*}
\pstwowar{|X_t|=k}{X_s=\alpha}=\sum_{|\beta|=k} \genbin{\beta}{\alpha}\frac{d_\beta}{d_\alpha}\left(1-e^{-t}\right)^{|\beta|-|\alpha|}\left(e^{-t}\right)^{|\alpha|+n}=\\
\left(1-e^{-t}\right)^{k-|\alpha|}\left(e^{-t}\right)^{|\alpha|+n}\sum_{|\beta|=k} \genbin{\beta}{\alpha}\frac{d_\beta}{d_\alpha}=
\binom{k+n-1}{k-|\alpha|}\left(1-e^{-t}\right)^{k-|\alpha|}\left(e^{-t}\right)^{|\alpha|+n}
\end{multline*}
(the last equality follows from Corollary \ref{cor:aux1}). Hence the above conditional probability depends on $\alpha=(\alpha_1,\ldots,\alpha_r)\in\Lambda$ only through $|\alpha|=\alpha_1+\ldots+\alpha_r$, so
\begin{equation}\label{8may}
\pstwowar{|X_t|=k}{X_s}=\pstwowar{|X_t|=k}{|X_s|}.
\end{equation}
Therefore the transitions of $(|X_t|)_t$ are the same as the transitions of the Yule process with immigration from Example \ref{ex:ordbincoeff}. Since $(|X_t|)_t$ is a Markov process with respect to the natural filtration of $(X_t)_t$, from \eqref{8may} it also follows that $(|X_t|)_t$ is Markov with respect to its own natural filtration as well. This proves the part of the proposition about birth processes. 

In order to prove the remaining part about death processes, one can argue in the same manner, using Lemma \ref{lem:sum} instead of Corollary \ref{cor:aux1} along the way.
\end{proof}

\section{Examples}\label{s:examples}

In this section we provide explicit computations of the generalized binomial coefficients and the corresponding transition rates for some multiplicity free $K$-actions 
for proper subgroups $K$ of $U(V)$ (the case of $U(V)$ acting on $V$ was treated in Example \ref{ex:ordbincoeff}, in the previous section).

In Example \ref{ex:ordbincoeff}, the irreducible subspaces in $\zesp[V]$ were the full subspaces $P_m^{\textrm{h}}(\zesp^n)$. The 
first example we present here is in a sense antipodal, as its irreducible subspaces are one-dimensional. 

\begin{example}\label{ex:ntorus}
Let $V=(\zesp^n,\iloskal{\cdot}{\cdot}_{\zesp^n})$ and $K$ be the $n$-torus
\begin{equation*}
K=T(n)=U(1)^n=\left\{\textrm{diag}(e^{i\theta_1},\ldots,e^{i\theta_n}):(\theta_1,\ldots \theta_n)\in\rzecz^n\right\}.
\end{equation*}
Under the usual action $k.z=kz$ ($k\in K, z\in V$) of $K$, the ring $\zesp[V]$ decomposes into $T(n)$-invariant one-dimensional subspaces $\zesp z^{\mu}$ with $\mu=(\mu_1,\ldots,\mu_n)\in\nat^n$
\begin{equation*}
\zesp[V]=\bigoplus_{\mu\in\nat^n} \zesp z^{\mu}
\end{equation*} 
so $\Lambda=\nat^n$. Clearly, \eqref{plambda} and \eqref{basis} imply that the invariant polynomials are $p_\mu(z)=\prod_i |z_i|^{2\mu_i}/\mu_i!$. This time $p_\mu$'s are products of functions of $|z_i|^2$ (the Gaussian kernel as well), so it is not difficult to notice that the task of orthogonalizing the sequence $(p_\mu)_\mu$ reduces to the orthogonalization of the monomials $(x^k)_k$ in $L^2([0,\infty),e^{-x})$, and that, consequently, one arrives at $q_\mu(z)=\prod_i  L^{(0)}_{\mu_i}\left(|z_i|^2\right)$. Since
\begin{multline*}
q_\mu(z)=\prod_{j=1}^n\sum_{k=0}^{\mu_j}\binom{\mu_j}{k}\frac{(-|z_j|)^{2k}}{k!}
=\sum_{\substack{\kappa=(\kappa_1,\ldots,\kappa_n)\in\nat^n\\ \kappa_i\le\mu_i}} \binom{\mu_1}{\kappa_1}\ldots\binom{\mu_n}{\kappa_n}
\prod_{j=1}^n\frac{(-|z_j|)^{2\kappa_j}}{\kappa_j!}\\
=\sum_{\substack{\kappa=(\kappa_1,\ldots,\kappa_n)\in\nat^n\\ \kappa_i\le\mu_i}}\binom{\mu_1}{\kappa_1}\ldots\binom{\mu_n}{\kappa_n}(-1)^{|\kappa|} p_\kappa(z),
\end{multline*}
formula \eqref{qinp} implies that the generalized binomial coefficients for the usual action of $T(n)$ on $V=\zesp^n$ are
\begin{equation}\label{genbintorus}
\genbin{\mu}{\kappa}=\prod_{j=1}^n\binom{\mu_j}{\kappa_j}
\end{equation}
when $\kappa_j\le\mu_j$ for $j=1,\ldots,n$, and $\genbin{\mu}{\kappa}=0$ otherwise. This implies, in particular, that if $|\mu|=|\kappa|+1$ and the generalized binomial coefficient \eqref{genbintorus} does not vanish, then there is exactly one $i\in\{1,\ldots,n\}$ with $\mu_i=\kappa_i+1$ and for remaining $j\ne i$, $\mu_j=\kappa_j$. Hence, from \eqref{q+} and Theorem \ref{thm:semibirth}, it follows that if birth rates of a pure birth process on $\nat^n$ are
\begin{equation}\label{q+torus}
q^+(\alpha,\beta)=\begin{cases}
\alpha_{i}+1,& \beta_i=\alpha_i+1\textrm{ and } \alpha_j=\beta_j\ \forall j\ne i \\
-n-\sum_{i=1}^n \alpha_i,& \alpha_i=\beta_i\textrm{ for all } i=1,\ldots,n,\\
0,& \textrm{otherwise},
\end{cases}
\end{equation}
then the semigroup of the process is
\begin{eqnarray}
p_t^+(\alpha,\beta)&=&\left(\prod_{j=1}^n \binom{\beta_j}{\alpha_j}\right)[1-\exp{(-t)}]^{\sum_{i=1}^n(\beta_i-\alpha_i)}\exp\left[-t\left(n+\sum_{i=1}^n\alpha_i\right)\right]\nonumber \\
&=& \prod_{j=1}^n \binom{\beta_j}{\alpha_j}[1-\exp(-t)]^{\beta_j-\alpha_j}[\exp(-t)]^{1+\alpha_j},\label{pt+torus}
\end{eqnarray}
if $ \alpha_i\le\beta_i\ \forall i=1,\ldots,n$, and $p_t^+(\alpha,\beta)=0$ in all other cases. Analogously, \eqref{q-} and Theorem \ref{thm:semibirth} provide the semigroup
\begin{eqnarray}
p_t^-(\alpha,\beta)&=&\left(\prod_{j=1}^n\binom{\alpha_j}{\beta_j}\right)\exp\left(-t\sum_{i=1}^n\beta_i\right)[1-\exp{(-t)}]^{\sum_{i=1}^n(\alpha_i-\beta_i)}\nonumber\\
&=&
\prod_{j=1}^n\binom{\alpha_j}{\beta_j}[\exp(-t)]^{\beta_j}[1-\exp(-t)]^{\alpha_j-\beta_j},\label{pt-torus}
\end{eqnarray}
when $\alpha_i\ge\beta_i\ \forall i=1,\ldots,n$ (and $p_t^-(\alpha,\beta)=0$ in all other cases), of a pure death process on $\nat^n$  with death rates of the form
\begin{equation}\label{q-torus}
q^-(\alpha,\beta)=
\begin{cases}
\alpha_i,& \alpha_i=\beta_i+1 \textrm{ and } \alpha_j=\beta_j\ \forall j\ne i \\
-\sum\alpha_i,& \alpha_i=\beta_i\textrm{ for all } i=1,\ldots,n,\\
0,& \textrm{otherwise}.
\end{cases}
\end{equation}

Obviously, \eqref{pt+torus} is the semigroup of $n$ independent copies of Yule process with both the individual birth and immigration rates equal $1$, and \eqref{pt-torus} is the semigroup of $n$ independent copies of a pure death process with the death rate $1$
(see also the last paragraph of the next example for an additional interpretation).

\end{example}
A modification of Example \ref{ex:ntorus} is given by
\begin{example}\label{ex:torussym}
Let $S_n$ denote the symmetric group. Consider $K=S_n\ltimes T(n)$ acting on $V=\zesp^n$ by
\begin{equation*}
\left(\sigma,(e^{i\theta_1},\ldots,e^{i\theta_n})\right).(z_1,\ldots,z_n)=\left(e^{i\theta_1} z_{\sigma(1)},\ldots,e^{i\theta_n} z_{\sigma(n)}\right),
\end{equation*}
for $\sigma\in S_n, (\theta_1,\ldots ,\theta_n)\in\rzecz^n$.
In \cite[Section 5.3]{BensonRatcliff2009} it was proved that the irreducible subspaces for this action are
\[
P_\lambda=\textrm{span}\left\{z^{\sigma.\lambda}:\sigma\in S_n\right\},
\]
where $\lambda\in\nat^n$ and $\sigma.\lambda=(\lambda_{\sigma(1)},\dots,\lambda_{\sigma(n)})$. Hence
\[
\zesp[V]=\bigoplus_{\lambda} P_\lambda,
\]
where the sum is taken over $\lambda$ in distinct $S_n$-orbits. 
Clearly,
\begin{equation}\label{dimplambda1}
d_\lambda=\dim P_\lambda=\frac{n!}{\prod_i \lambda[i]!},
\end{equation} 
where $\lambda[i]$ is the number of $\lambda_j$ in $\lambda=(\lambda_1,\ldots,\lambda_n)$ equal $i$ (so the frequency representation of $\lambda$ \cite[page 1]{Andrews98} is 
 $\lambda=(0^{\lambda[0]}1^{\lambda[1]} 2^{\lambda[2]} 3^{\lambda[3]}\ldots)$). 

The family of $K$-invariant polynomials consists of
\[
p_\lambda=\frac{1}{d_\lambda}\sum_{\mu\in S_n.\lambda}\frac{|z^2|^\mu}{\mu!}=\frac{1}{n!}\sum_{\sigma\in S_n}\frac{|z^2|^{\sigma(\lambda)}}{\lambda!}
\]
(obviously, by $|z^2|^\mu$ we denote $z^\mu\overline{z^\mu}$). Since
\[
\int_V \sum_{\mu\in S_n.\lambda} |z^2|^\mu\exp\left(-\sum_{i=1}^n|z_i|^2\right)\textrm{d}z=\sum_{\mu\in S_n.\lambda}
\prod_{i=1}^n\left(\int_\zesp|z_i|^{2\mu_i}\exp(-|z_i|^2)\textrm{d}z_i\right),
 \]
the task of orthogonalizing the sequence $(p_\lambda)_\lambda$ is similar to the one in Example \ref{ex:ntorus}, and it yields
\begin{equation}\label{qlambda1}
q_\lambda=\frac{\prod_i \lambda[i]!}{n!}\sum_{\mu\in S_n.\lambda} \prod_{i=1}^n  L^{(0)}_{\mu_i}\left(|z_i|^2\right).
\end{equation}
The factor in front of the sum in \eqref{qlambda1} appears due to the required normalization $q_\lambda(0)=1$, and in fact equals the inverse of $\dim P_\lambda$.

It is not difficult to see that
\begin{multline*}
q_\lambda=\frac{1}{n!}\sum_{\sigma\in S_n}\prod_{i=1}^n  L^{(0)}_{\sigma(\lambda)_i}\left(|z_i|^2\right)=
\frac{1}{n!}\sum_{\sigma\in S_n}\sum_{\substack{\mu:\\ 0\le\mu_i\le\sigma(\lambda)_i, \\ i=1,\ldots,n}} \frac{(-1)^\mu}{\mu!}
\left[\prod_{i=1}^n\binom{\sigma(\lambda)_i}{\mu_i}\right]|z^2|^\mu=\\
=\frac{1}{n!}\sum_{|\mu|\le|\lambda|}\frac{(-1)^\mu}{\mu!}|z^2|^\mu
\sum_{\substack{\sigma\in S_n:\\ 0\le\mu_i\le\sigma(\lambda)_i, \\ i=1,\ldots,n}} \prod_{i=1}^n\binom{\sigma(\lambda)_i}{\mu_i}.
\end{multline*} 
In order to compute the generalized binomial coefficients associated with the action of $S_n\ltimes T(n)$ on $\zesp^n$, we will write $q_\lambda$ as a linear combination of $p_\mu$). 
Due to the decomposition of $\zesp[V]$ in this case, both sums $\sum_{|\mu|\le|\lambda|}$ below are additionally taken over $\mu$ in distinct $S_n$-orbits:
\begin{multline*}
q_\lambda=\frac{1}{n!}\sum_{|\mu|\le|\lambda|}\frac{(-1)^\mu}{\mu!}\sum_{\substack{\sigma\in S_n:\\ 0\le\mu_i\le\sigma(\lambda)_i, \\ i=1,\ldots,n}} \prod_{i=1}^n\binom{\sigma(\lambda)_i}{\mu_i}\frac{1}{\prod_j \mu[j]!}\sum_{\sigma\in S_n} |z^{\sigma(\mu)}|^2=\\
\sum_{|\mu|\le|\lambda|}(-1)^\mu\frac{1}{\prod_j \mu[j]!}\sum_{\substack{\sigma\in S_n:\\ 0\le\mu_i\le\sigma(\lambda)_i, \\ i=1,\ldots,n}} \prod_{i=1}^n\binom{\sigma(\lambda)_i}{\mu_i}p_\mu.
\end{multline*}
Therefore the generalized binomial coefficients for the action of $K=S_n\ltimes T(n)$ on $\zesp^n$ are
\begin{equation}\label{genbinsymtor}
\genbin{\lambda}{\mu}=\frac{1}{\prod_j \mu[j]!}\sum_{\substack{\sigma\in S_n:\\ 0\le\mu_i\le\sigma(\lambda)_i, \\ i=1,\ldots,n}} \prod_{i=1}^n\binom{\sigma(\lambda)_i}{\mu_i}=\frac{1}{\prod_j \mu[j]!}\sum_{\substack{\sigma\in S_n:\\ 0\le\sigma(\mu)_i\le\lambda_i, \\ i=1,\ldots,n}} \prod_{i=1}^n\binom{\lambda_i}{\sigma(\mu)_i},
\end{equation}
as
\[
\prod_{i=1}^n \binom{\sigma(\lambda)_i}{\mu_i}=\prod_{i=1}^n \binom{\lambda_i}{\sigma^{-1}(\mu)_i}.
\]
A slightly incorrect version of \eqref{genbinsymtor} was first obtained in \cite{BensonRatcliff2009}.

Guided by Section \ref{s:birth-death-Kaction}, we are now in a position to use \eqref{genbinsymtor} along with \eqref{q+} and \eqref{q-} to define some pure birth and death processes; we may and will in fact view them as processes on partitions of the length not exceeding $n$ as in every $S_n$-orbit there is exactly one such partition. Recall that a partition is a weakly decreasing sequence of non-negative integers $\lambda=(\lambda_1,\ldots,\lambda_s)$, called parts; the length of $\lambda$ is the number of non-zero parts, while $\lambda_1+\ldots+\lambda_s$ is called the weight of $\lambda$, which in our setting equals the degree of homogeneity $|\lambda|$ of the polynomials from $P_\lambda$.

As in the previous example, if $|\lambda|=|\mu|+1$ and the generalized binomial coefficient \eqref{genbinsymtor} does not vanish, then there is exactly one $i\in\{1,\ldots,n\}$ with $\lambda_i=\mu_i+1$ and for remaining $j\ne i$, $\lambda_j=\mu_j$. It is worth recalling now that a partition $\lambda=(\lambda_1,\ldots,\lambda_s)$ can be visualized by Young (or Ferrer) diagram, which is an array of $\sum_i \lambda_i$ boxes having $s$ left-justified rows with row $i$ containing $\lambda_i$ boxes for $i=1,\ldots,s$. The partitions can be identified with their Young diagrams. The above-mentioned remark can be restated in the language of Young diagrams in the following way: if \eqref{genbinsymtor} is non-zero and $|\lambda|=|\mu|+1$ then $\lambda$ is obtained from $\mu$ by appending a single box. It is not difficult to observe more generally that if $\lambda\not\supset\mu$ (in the sense of Young diagrams) then the generalized binomial coefficient \eqref{genbinsymtor} vanishes.

Let us compute the transition rates \eqref{q+} and \eqref{q-} originating from the action of $K=S_n\ltimes T(n)$. Assume first that $\beta$ arises from $\alpha$ by replacing $i$ by $i+1$, that is: 
\begin{equation}\label{itoi+1}
\beta[i]=\alpha[i]-1, \beta[i+1]=\alpha[i+1]+1, \beta[j]=\alpha[j] \textrm{ for }j\ne i,i+1.
\end{equation} 
In such case 
\[
\genbin{\beta}{\alpha}=\frac{1}{\prod_j (\alpha[j])!}(i+1)(\alpha[i+1]+1)!(\alpha[i])!\prod_{j<i}(\alpha[j])!\prod_{j>i+1}(\alpha[j])!.
\]
Combined with \eqref{dimplambda1}, it gives the following birth rates:
\begin{equation}\label{q+symtor}
q^+(\alpha,\beta)=
\begin{cases}
(i+1)\alpha[i],& \textrm{if }\eqref{itoi+1}\textrm{ holds,}\\
-n-|\alpha|,& \alpha=\beta,\\
0,& \textrm{otherwise.}
\end{cases}
\end{equation}
Theorem \ref{thm:semibirth} asserts that the semigroup $(p^+_t)_t$ of transition probabilities of the birth process with the intensities \eqref{q+symtor}
is
\[
p^+_t(\alpha,\beta)=
\begin{cases}
\frac{1}{\prod_i \beta[i]!}\sum_{\sigma\in S_n}\prod_{i=1}^n\Bigl[\binom{\beta_i}{\sigma(\alpha)_i}(1-e^{-t})^{\beta_i-\sigma(\alpha)_i}(e^{-t})^{\sigma(\alpha)_i+1}\Bigr],& |\beta|\ge|\alpha|\\
0,& \textrm{otherwise.}
\end{cases}
\]

Now, suppose that $\beta$ arises from $\alpha$ by replacing $i$ with $i-1$, that is
\begin{equation}\label{itoi-1}
\beta[i-1]=\alpha[i-1]+1, \beta[i]=\alpha[i]-1, \beta[j]=\alpha[j] \textrm{ for }j\ne i-1,i.
\end{equation}
Then
\begin{multline*}
\genbin{\alpha}{\beta}=
\frac{1}{\prod_{j<i-1}(\alpha[j])!(\alpha[i-1]+1)!(\alpha[i]-1)!\prod_{j>i}(\alpha[j])!}\\
\times i(\alpha[i])!(\alpha[i-1+1])!\prod_{j<i-1}(\alpha[j])!\prod_{j>i}(\alpha[j])!
=i\cdot\alpha[i].
\end{multline*}
Therefore, by Theorem \ref{thm:semibirth}, if transition rates of a pure death process are
\begin{equation}\label{q-symtor}
q^{-}(\alpha,\beta)=
\begin{cases}
i\cdot\alpha[i],& \textrm{if }\eqref{itoi-1}\textrm{ holds,}\\
-|\alpha|,& \alpha=\beta,\\
0,& \textrm{otherwise,}
\end{cases}
\end{equation}
then its semigroup
\[
p^-_t(\alpha,\beta)=
\begin{cases}
\frac{1}{\prod_i \beta[i]!}\sum_{\sigma\in S_n}\prod_{i=1}^n\Bigl[\binom{\beta_i}{\sigma(\alpha)_i}(1-e^{-t})^{\alpha_i-\sigma(\beta)_i}(e^{-t})^{\sigma(\beta)_i}\Bigr],& |\alpha|\ge|\beta|,\\
0,& \textrm{otherwise.}
\end{cases}
\]

When viewed as some random evolutions of Young diagrams, the birth and death processes obtained in this example gain a clear interpretation. A birth means adding a box to an existing diagram of a partition $\alpha$; \eqref{q+symtor} implies that the probability of adding a box to a row with $i$ boxes equals $((i+1)\alpha[i])/(n+|\alpha|)$, so it is proportional to the number of boxes in this row (plus $1$), as well as to the number of rows in the diagram of $\alpha$ having the same number $i$ of boxes. Since the state space is a set of Young diagrams, if a box is added to one of multiple rows of the same length, it is in fact added to the first of these rows. One needs to keep in mind that if the dynamics of the process have resulted in creating $n$ rows (with strictly positive numbers of boxes), it is impossible to create any more rows and new boxes will keep being added to the existing rows.

Analogously, from \eqref{q-symtor} it follows that the probability of a death (deleting a box from a row with $i$ boxes) equals $(i\alpha[i])/|\alpha|$, so again, it depends not only on the number of boxes but also on the number of rows of the same length. If a box is deleted from one of multiple rows of the same length, it is in fact deleted from the last of these rows.

It is now worth contrasting these processes with the ones considered in Example \ref{ex:ntorus}. The latter live on $\nat^n$ (not partitions), but we can still associate an array of boxes with every $\kappa=(\kappa_1,\ldots,\kappa_n)\in\nat^n$ so that $i$-th row in the array contains $\kappa_i$ boxes, $i=1,\ldots,n$. Under this identification, \eqref{q+torus} and \eqref{q-torus} imply that the probability of adding (or deleting) a box is only proportional to the number of boxes in a row. If there are multiple rows with the same number of boxes, each can be changed with the same probability, regardless of the position in the array.

To get yet another insight on the dynamics governed by \eqref{q+symtor}, one can view the Young diagram of a partition $\alpha$ (with the length not greater than $n$) as the diagram with exactly $n$ rows (some of the last rows will perhaps have no boxes), and then add one box to each of the $n$ rows, thus creating a ``virtual diagram'' with $n+|\alpha|$ boxes. Next, \eqref{q+symtor} implies that one picks a box uniformly at random from the virtual diagram, and adds one box to the original diagram of $\alpha$, placing it in the first row with the same number of boxes as the row from which the randomly picked box came.

A similar viewpoint can be offered for the dynamics of the death process given by \eqref{q-symtor}. In this case there is no need for the ``virtual diagram'' and one directly picks a box uniformly at random from the Young diagram of $\alpha$. Now, \eqref{q-symtor} means that one box is removed from the last row containing the same number of boxes as the row with the randomly chosen box.

\end{example}

In the next example we examine four multiplicity free actions in parallel. 

\begin{example}\label{ex:fiveactions}
Below, $M_{m,m}(\zesp)$ denotes the set of all $m\times m$ complex matrices, and $\textrm{SO}(\rzecz,m)$ stands for the group of the real orthogonal $m\times m$ matrices with determinant one. We consider:
\renewcommand{\labelenumi}{(\roman{enumi})}
\begin{enumerate}
	
	\item\label{action1}  $K=U(m)$ acting on the space $V=\symmc=\{z\in M_{m,m}(\zesp): z=z^t\}$ of $m\times m$ (with $m(m+1)/2=n$) complex symmetric matrices via
	\begin{equation*}
	k.z=k z k^{t},
	\end{equation*}
	
	\item $K=U(m)\times U(m)$ acting on $V=M_{m,m}(\zesp)$ (with $m^2=n$) via
	\begin{equation*}
	(k_1,k_2).z=k_1 z k_2^{-1},
	\end{equation*}
	
	\item $K=U(2m)$ acting on the space $V=\textrm{Skew}(m,\zesp)=\{z\in M_{2m,2m}(\zesp):z=-z^t\}$ of complex skew-symmetric $m\times m$ matrices (with $m(2m-1)=n$) by the formula from (i),
	
	\item $K=\textrm{SO}(\rzecz,n)\times T(1)$ acting on $V=\zesp^n$ via
	\begin{equation*}
	(k_1,k_2).z=k_2 k_1 z,\quad k_1\in\textrm{SO}(\rzecz,n), k_2\in\zesp, |k_2|=1, z\in\zesp^n
	\end{equation*}

\end{enumerate}	

Those actions have a common origin - they all arise in connection with Hermitian symmetric spaces and Jordan algebras. We will now very briefly sketch the link with Jordan algebras (all necessary definitions can be found in the monograph \cite{farautkoranyi} by Faraut and Kor\'{a}nyi, or, in an abbreviated version, in \cite{matysiakswieca2}).

If $W$ is a simple complex Jordan algebra, then 
$J$ defined as the identity component of $\textrm{Str}(W)\cap U(W)$, where $U(W)$ stands for the unitary group of $W$, and $\textrm{Str}(W)$ for the structure group of $W$, acts multiplicity free on $W$.  Up to isomorphism, there are five families of simple complex
Jordan algebras. Four of them coincide with the four spaces $V$ listed above, and the groups $K$ giving the actions (i)--(iv) coincide with the groups $J$ acting multiplicity free on $W$ (see the first three columns of Table \ref{tab:1}). 

\begin{table}[h]
	\begin{tabular}{cccccc}
		\hline
		Action &	$W$ & $J$ & $r$ &  $d$ & $\theta$ \\
		\hline
		(i) &		$\symmc$ &  $U(m)$ & m & 1 & $1/2$\\
		(ii) &		$M(m,\zesp)$ & $U(m)\times U(m)$ & m & 2 & $1$\\
		(iii) &		$\textrm{Skew}(2m,\zesp)$  & $U(2m)$ & m & 4 & $2$\\
		(iv)&		$\zesp^n$ &  $SO(n)\times\mathbb{T}$ & 2 & $n-2$ & $(n-2)/2$\\
		\hline
	\end{tabular}
	\caption{}
	\label{tab:1}
\end{table}

The set $\Lambda$ that parametrizes the decomposition \eqref{decompos} of $\zesp[W]$  turns out to be the set of all partitions
of length not greater than the rank of the Jordan algebra $W$. This means that the ranks of the multiplicity free actions (i)--(iv) are equal to the ranks of the corresponding Jordan algebras (given in the fourth column of Table \ref{tab:1}). Furthermore, the dimensions of the $K$-irreducible subspaces $P_\lambda$ were computed by Upmeier \cite[Lemma 2.6]{Upmeier1983} (see also \cite[p. 315]{farautkoranyi}): 
\begin{multline}\label{upmeier}
d_\lambda=\prod_{1\le p<q\le r}\frac{\lambda_p-\lambda_q+\frac{d}{2}(q-p)}{\frac{d}{2}(q-p)}
\frac{B\left(\lambda_p-\lambda_q,\frac{d}{2}(q-p-1)+1\right)}{B\left(\lambda_p-\lambda_q,\frac{d}{2}(q-p+1)\right)}=\\
\prod_{p=1}^r \frac{\Gamma\left(\frac{d}{2}\right)}{\Gamma\left(\frac{d}{2} p\right)\Gamma\left(\frac{d}{2}(p-1)+1\right)}\\
\times 
\prod_{1\le p<q\le r}\left(\lambda_p-\lambda_q+\frac{d}{2}(q-p)\right)
\frac{\Gamma\left(\lambda_p-\lambda_q+\frac{d}{2}(q-p+1)\right)}{\Gamma\left(\lambda_p-\lambda_q+\frac{d}{2}(q-p-1)+1\right)}.
\end{multline}
The parameter $d$ appearing in \eqref{upmeier} denotes the Peirce constant of the Jordan algebra (to be found in the fifth column of Table \ref{tab:1}).

An important role in the theory of Jordan algebras is played by the spherical polynomials $\sph{\lambda}$, which are certain homogeneous polynomials of degree $|\lambda|$. The spherical polynomials corresponding to (ii) are, up to a normalizing constant, Schur functions (see e.g. \cite{macdonaldbook}); the (zonal) spherical polynomials associated with  (i) play an important role in multivariate statistical theory (see \cite{muirheadbook}). In the case of (iv), the spherical polynomials can be expressed in terms of the ultraspherical polynomials. For any 
simple Euclidean Jordan algebra, $\sph{\lambda}(x)$ regarded as (symmetric) functions of the eigenvalues of $x$, are some special cases of Jack polynomials $J_\lambda(\cdot;\theta)$. Recall that Jack polynomials $J_\lambda(z_1,\ldots,z_N;\theta)$, indexed by partitions (Young diagrams) $\lambda$ of the length $N$ and a positive parameter $\theta$, can be defined as the unique symmetric polynomial eigenfunctions of a certain differential operator with, the leading term being the monomial symmetric function $z^\lambda$. For the relevant properties of these polynomials we refer to Macdonald's book \cite{macdonaldbook}, with a remark that Macdonald uses a parameter $\alpha$ given by our $\theta^{-1}$. The special values of $\theta$ for which Jack polynomials are the spherical polynomials in a Jordan algebra are linked with the algebra's Peirce constant $d$ by $\theta=d/2$ (see the last column of Table \ref{tab:1}).

In the theory of Jordan algebras certain generalized binomial coefficients were introduced  as the coefficients of the expansion of spherical polynomial $\sph{\lambda}$ (evaluated at a point shifted by the identity element) into spherical polynomials of degrees not higher than $|\lambda|$ (see p. 343 in \cite{farautkoranyi}). (It is worth noting that the generalized binomial coefficients for $\symmc$ were studied independently by Bingham \cite{bingham1974}.)

Lassalle \cite[p. 254]{Lassalle1990} defined another version of generalized binomial coefficients in completely analogous fashion, using Jack polynomials $J_\lambda(\cdot;\theta)$ instead of $\sph{\lambda}$ (see also \cite{Kaneko1993} and \cite{OkounkovOlshanski1997}). Due to compatibility of the definitions, Lassalle's coefficients agree with the ones from corresponding Jordan algebras if $\theta$ takes values listed in the last column of Table \ref{tab:1}. Since Yan proved in \cite{yan_unpub} that the generalized binomial coefficients from Jordan algebras agree with the ones originating from appropriate actions (i)--(iv), we conclude that Lassalle's coefficients coincide with their counterparts for actions (i)-(iv).

Moreover, Lassalle (\cite[p. 320]{Lassalle1998}, see also \cite[Th\'{e}or\`{e}me 5]{Lassalle1990}) found the following explicit formula, which, in view of the above remarks, holds true for the generalized binomial coefficients associated with the multiplicity actions considered in this example, provided $\theta$ takes relevant value: if $\lambda$ and $\mu$ are partitions of the length not greater than $r$, and there exists $i\in\{1,\ldots,r\}$ such that $\mu_i=\lambda_i-1$ and $\mu_j=\lambda_j$ for $j\ne i$, then
\begin{equation}\label{lassalle}
{\genbin{\lambda}{\mu}}=[\lambda_i+\theta(r-i)]\prod_{\substack{j=1\\j \ne i}}^r
\frac{\lambda_i-\lambda_j+\theta(j-i-1)}{\lambda_i-\lambda_j+\theta(j-i)}.
\end{equation}

Okounkov and Olshanski \cite{OkounkovOlshanski1997} showed, using a result of Knop and Sahi \cite{KnopSahi1996}, that the generalized binomial coefficients $\genbin{\lambda}{\mu}$ associated with Jack polynomials (for all values of the parameter $\theta$) vanish unless $\mu\subset\lambda$. This fact, together with formulas \eqref{lassalle} and \eqref{upmeier}, allow us to explicitly compute the transition rates \eqref{q+} and \eqref{q-} for the pure birth and pure death processes associated with the actions (i)-(iv). 

At least two cases deserve special mention. The first one deals with action (ii), and it is the value of $\theta=1$ what accounts for the simpler formulas in this case compared to the remaining actions. One quickly gets the birth rates
\begin{equation*}
q^{+}(\alpha,\beta)=\begin{cases}
(\alpha_i+m-i+1)\prod_{j\ne i}\frac{\alpha_i-\alpha_j+j-i+1}{\alpha_i-\alpha_j+j-i},& \beta_i=\alpha_i+1 \textrm{ and } \beta_j=\alpha_j\ \forall j\ne i,\\
-m^2-\sum_j \alpha_j,& \beta=\alpha,\\
0,& \textrm{otherwise,}
\end{cases}
\end{equation*}
and the death rates
\begin{equation*}
q^{-}(\alpha,\beta)=
\begin{cases}
(\alpha_i+m-i)\prod_{j\ne i}\frac{\alpha_i-\alpha_j+j-i-1}{\alpha_i-\alpha_j+j-i}, & \beta_i=\alpha_i-1 \textrm{ and } \beta_j=\alpha_j\ \forall j\ne i,\\
-\sum_j \alpha_j,& \beta=\alpha,\\
0, & \textrm{otherwise.}
\end{cases}
\end{equation*}
Observe that similar, although not identical formulas appear in \cite{BorodinOlshanski2012} (see also \cite{Olshanski2015}) where they provide the jump rates for some birth and death processes on the dual object of the unitary group.

The second case worth displaying is the one for action (iv). Here the low rank $r=2$ of the action is the simplifying factor. It is straightforward to verify the formulas for the generalized binomial coefficients
\begin{equation*}
\genbin{(\alpha_1,\alpha_2)}{(\alpha_1-1,\alpha_2)}=\frac{(\alpha_1-\alpha_2)(\alpha_1+\theta)}{\alpha_1-\alpha_2+\theta},\quad
\genbin{(\alpha_1,\alpha_2)}{(\alpha_1,\alpha_2-1)}=\frac{\alpha_2(\alpha_1-\alpha_2+2\theta)}{\alpha_1-\alpha_2+\theta},
\end{equation*}
the birth rates
\begin{equation*}
q^+(\alpha,\beta)=
\begin{cases}
\frac{(\alpha_1+1+\theta)(\alpha_1-\alpha_2+2\theta)}{\alpha_1-\alpha_2+\theta},& \beta=(\alpha_1+1,\alpha_2),\\
\frac{(\alpha_2+1)(\alpha_1-\alpha_2)}{\alpha_1-\alpha_2+\theta},& \beta=(\alpha_1,\alpha_2+1),\\
-n-\alpha_1-\alpha_2,& \beta=\alpha,\\
0,& \textrm{otherwise}
\end{cases}
\end{equation*}
and the death rates
\begin{equation*}
q^-(\alpha,\beta)=
\begin{cases}
\frac{(\alpha_1+\theta)(\alpha_1-\alpha_2)}{\alpha_1-\alpha_2+\theta},& \beta=(\alpha_1-1,\alpha_2),\\
\frac{\alpha_2(\alpha_1-\alpha_2+2\theta)}{\alpha_1-\alpha_2+\theta},& \beta=(\alpha_1,\alpha_2-1),\\
-\alpha_1-\alpha_2,& \beta=\alpha,\\
0,& \textrm{otherwise.}
\end{cases}
\end{equation*}
(Recall that here $\theta=(n-2)/2$.) See Remark \ref{rem:Biane} in Section \ref{s:heis} for a link of these intensities with Biane's work \cite{biane1998}.
\end{example}

\begin{remark}\label{rem:Jordan_examples}
As it was mentioned in Example \ref{ex:fiveactions}, the simple complex Jordan algebras were classified into five cases. Four (\emph{classical}) cases are presented in Table \ref{tab:1}. The fifth one is said to be \emph{exceptional} and is in fact the exceptional $27$-dimensional Albert algebra which can roughly be considered as a space $\textrm{Herm}(3,\mathbb{O})$ of $3\times 3$ Hermitian matrices on octonions. Although it follows from the general theory laid out in \cite{farautkoranyi} that there exists a multiplicity free action associated with the Albert algebra (involving one of the five exceptional Lie algebras from the Cartan--Killing classification of the complex simple Lie algebras), we excluded this case from our considerations due to its undoubtedly exceptional character.
\end{remark}

\begin{remark}
The decomposition for $\zesp[V]$ originating from action (iv) is given by the classical theory of spherical harmonics (see \cite{SteinWeiss1971}). The generalized binomial coefficients in this case were obtained (without appealing to Jack polynomials) by Ding \cite{Ding2003}.
\end{remark}

\section{Heisenberg group and quantum Ornstein-Uhlenbeck process}\label{s:heis}

Now we will explain a connection between the birth and death processes defined in Section \ref{s:birth-death-Kaction}, and the quantum Ornstein-Uhlenbeck processes introduced by Biane (\cite{biane1996} and \cite{biane1998}).

\subsection{Gelfand pairs associated with the Heisenberg group}\label{ss:gelfheis}

We identify the Heisenberg group $H$ with $V\times \rzecz$ with multiplication given by
\[
(z,w)(z^\prime,w^\prime)=(z+z^\prime,w+w^\prime+\MojeIm\iloskal{z^\prime}{z})
\]
for $z,z^\prime\in V$ and $w,w^\prime\in\rzecz$. The Heisenberg group is a non-abelian locally compact group, with the well-known representations. The irreducible, infinite dimensional unitary representations of $H$ can be realized on the Segal--Bargmann space $\mathcal{F}_a(V)$ ($a\in\rzecz\setminus\{0\}$) of holomorphic functions $\psi$ on $V$ such that 
\[
\iloskal{\psi}{\psi}_{|a|}<\infty
\]
(recall \eqref{fockinner}). The Bargmann-Fock representation $\pi_1$ of $H$ on $\mathcal{F}_1$ is given by
\begin{equation}\label{bargmannfock}
\left(\pi_1(z,w)\psi\right)(\zeta)=\exp\left(iw-\frac{\norma{}{z}^2}{2}-\iloskal{\zeta}{z}\right)\psi(\zeta+z).
\end{equation}
Using the scaling automorphisms (dilatations) $\delta_a$ of the Heisenberg group
\begin{equation}
\delta_{a}(z,w)=\begin{cases}
(az,a^2w),\ a>0,\\
(a\bar{z},-a^2w),\ a<0,
\end{cases}
\end{equation}
one can neatly define representations $\pi_a=\pi_1\circ\delta_a$ on $\mathcal{F}_a$ for any $a\in\rzecz\setminus\{0\}$. The Stone--von Neumann theorem implies that the Bargmann-Fock representations $\pi_a$ with $a\ne0$, along with some one dimensional representations, essentially (up to a unitary equivalence) exhaust the unitary dual of $H$.

The natural action of $U(V)$ on $V$ gives rise to a compact subgroup of the group $\Aut(H)$ of automorphisms of $H$, again denoted $U(V)$, via 
\begin{equation}\label{KActsOnH}
k.(z,w)=(k.z,w),\ k\in U(V), (z,w)\in H,
\end{equation}
therefore it makes sense to consider $K$-invariant functions on $H$ for any given compact subgroup of $U(V)$. The $K$-invariant $L^1$-functions on $H$ make a convolution subalgebra $L^1(H)^K$ of $L^1(H)$, and one says that $(K,H)$ is a \emph{Gelfand pair} if $L^1(H)^K$ is commutative. The following theorem, which is a special case of a result proved by Carcano in \cite{carcano}, provides the fundamental link between the Gelfand pairs built on Heisenberg groups and multiplicity free actions.

\begin{theorem}
The pair $(K,H)$ is a Gelfand pair if and only if the action of $K$ on $V$ is multiplicity free.
\end{theorem}

The significance of the Gelfand pairs in the context of this paper comes from the commutativity of the $\textrm{C}^\ast$-completion $\chk$ of the $*$-algebra $L^1(H)^K$. By the Gelfand-Naimark theorem, commutativity of $\chk$ implies 
isomorphism
\[
\chk\cong C_0 \left(\sigma\left(\chk\right)\right),
\]
where $\sigma\left(\chk\right)$ is the spectrum (the Gelfand space) of $\chk$. The Gelfand space $\sigma\left(\chk\right)$ can be identified (via integration) with the set $\Sigma(H)^K$ of all \emph{bounded $K$-spherical functions on $H$}, which can be defined in several ways, e.g. as the smooth $K$-invariant functions on $H$, taking value $1$ at $(0,0)$, being joint eigenfunctions for the differential operators on $H$ that are invariant under $K$-action and under the left action of $H$. 

In \cite{BJR92}, the bounded $K$-spherical functions were determined for all compact Lie subgroups of $U(V)$ for which $(K,H)$ is a Gelfand pair. Interestingly, there is a tight connection between the polynomials $\{q_\lambda\}$ for a multiplicity free action of $K\subset U(V)$ defined in subsection \ref{ss:polynomials}, and the elements of $\Sigma(H)^K$, where $H=V\times\rzecz$ is the Heisenberg group built on $V$. Namely, Benson, Jenkins and Ratcliff show in \cite{BJR92} that the functions $\varphi_{s,\lambda}$ (\emph{the spherical functions of the first kind} in the terminology of \cite{BJR92}), defined for $s\in\rzecz^\ast$ and $\lambda\in \Lambda$ by
\begin{equation}\label{spherfun}
\varphi_{s,\lambda}(z,w)=\exp\left[isw-\frac{|s|\norma{}{z}^2}{2}\right]{q}_\lambda\left(\sqrt{|s|}z\right), \quad (z,w)\in H,
\end{equation}
are bounded $K$-spherical functions on $H$ (in fact, \eqref{spherfun} is a rescaled version of the analogous formula from \cite{BJR92} as in \cite{BJR92} a slightly different definition of the group law in $H$ was considered; see Remark \ref{rem:conventions}). There exist some elements of $\Sigma(H)^K$ (\emph{the spherical functions of the second kind} in terminology of \cite{BJR92}) different from the ones defined by \eqref{spherfun} but they will not play a role in this paper.

\subsection{Non-commutative Brownian motion}\label{ss:quantBM}

It is well known that if $G$ is a locally compact group and $\pi$ is a unitary representation of $G$ on a Hilbert space $H_\pi$ then for any $v\in H_\pi$, the function $\phi:G\to\zesp$ defined as $\phi(g)=\iloskal{\pi(g)v}{v}_{H_\pi}$ is of positive type (in fact, any nonzero function of positive type on $G$ arises from a unitary representation in this way). In the context of the Heisenberg group, one can show using the Bargmann-Fock representation \eqref{bargmannfock} that for any positive $t$,
\[
(z,w)\mapsto \exp\left[-t\psi(z,w)\right]
\]
is of positive type provided that the mapping $\psi$ from $H$ to $\zesp$ is given by
\begin{equation}\label{psi}
\psi(z,w)=-iw+\frac{\norma{}{z}^2}{2}.
\end{equation}

Biane \cite{biane1996} used this fact and considered functions $Q_t:L^1(H)\to L^1(H)$, given for $t\ge0$ by the formula
\begin{equation}\label{Qt}
(Q_t f)(z,w)=\exp\left[-t\psi(z,w)\right]f(z,w)
\end{equation}
(the product above is the pointwise multiplication of functions on $H$). Biane (see \cite[1.2.1 Proposition]{biane1996}) proved that \eqref{Qt} defines the semigroup $(Q_t)_t$ of completely positive contractions on $L^1(H)$ that extends in a unique way to the semigroup of completely positive contractions on the group $C^\ast$-algebra $C^\ast(H)$ of the Heisenberg group. For the extension we will use the same notation $(Q_t)_t$ as for the original semigroup on $L^1(H)$. As it is explained in \cite[Chapter 5]{biane2008}, $(Q_t)_t$ is the semigroup of a non-commutative Brownian motion on the unitary dual of $H$.

\subsection{Quantum Ornstein-Uhlenbeck processes}\label{ss:ornstuhlen}

Recall that if $(B_t)_t$ is classical Brownian motion starting from zero, then 
\[
(\exp(t/2)B_{\exp(-t)})_t\quad\textrm{and}\quad(\exp(-t/2)B_{\exp(t)})_t
\]
are Ornstein-Uhlenbeck processes. For the semigroup $(Q_t)_t$, one can analogously put
\begin{equation*}
\widetilde{R}_t^+=\delta_{\exp(t/2)}\circ Q_{1-\exp(-t)},\quad \widetilde{R}_t^-=\delta_{\exp(-t/2)}\circ Q_{\exp(t)-1}
\end{equation*}
(here $\delta_a$ stands for an operator on $L^1(H)$ induced by the dilation $\delta_a:H\to H$ via $L^1(H)\ni f\mapsto f\circ\delta_a\in L^1(H)$). It is not difficult to check that $(\widetilde{R}^\pm_t)_t$ are semigroups of completely positive contractions of $\cgw(H)$ (see \cite[p. 92]{biane1996}).

Every unitary representation $\pi_t:H\to U\left(\mathcal{F}_1\right)$ of the Heisenberg group determines the $C^\ast$-algebras mapping $\pi_t:\cgw(H)\to\mathcal{B}(\mathcal{F}_1)$ (with $\mathcal{B}$ denoting the algebra of bounded operators). In fact (see \cite{Lee1977}), 
\[
\pi_t:\cgw(H)\to\mathcal{K}(\mathcal{F}_1),
\]
where $\mathcal{K}$ stands for the algebra of compact operators. The following Proposition \ref{prop:biane312}, along with the above remarks, motivated Biane to name the semigroups $(R_t^\pm)_t$ appearing in the proposition, \emph{the quantum Ornstein-Uhlenbeck semigroups}.

\begin{proposition}[\cite{biane1996}, Prop. 3.1.2, \cite{biane1998}, Lemma 2.6]\label{prop:biane312}
	There exist semigroups $(R^\pm_t)_t$ of completely positive contractions on $\mathcal{K}\left(\mathcal{F}_1\right)$ such that
	\begin{eqnarray*}
		\pi_1\circ \widetilde{R}_t^+&=&R_t^+\circ\pi_1,\\
		\pi_{-1}\circ \widetilde{R}_t^-&=&R_t^-\circ\pi_{-1}.
	\end{eqnarray*}  
\end{proposition}
It is known (see e.g. \cite[p. 62]{biane1998}) that the image $\pi_{\pm1}\left(\chk\right)\subset\mathcal{K}(\mathcal{F}_1)$ is the commutative 
$\cgw$-algebra generated by the orthogonal projections onto the subspaces $P_\lambda$, $\lambda\in\Lambda$, appearing in \eqref{decompos}, and that the spectrum of the algebra $\pi_{\pm1}\left(\chk\right)$ is homeomorphic to $\Lambda$ (with discrete topology). Since from Proposition \ref{prop:biane312} it follows that $\pi_{\pm1}\left(\chk\right)$ is $R_t^\pm$-invariant sub-$\cgw$-algebra of $\mathcal{K}(\mathcal{F}_1)$, we deduce that restricting $(R^\pm_t)_t$ to $\pi_{\pm1}\left(\chk\right)$ one gets a classical Markov process on the countable state space $\Lambda$.

\subsection{The connection}\label{ss:connection}

Let $(p_t^\pm)_t$ be the semigroups of transition probabilities on $\Lambda$ originating from the restriction of $(R^\pm_t)_t$, respectively. We purposely use here the same notation $(p_t^\pm)_t$ as in Section \ref{s:birth-death-Kaction} for the semigroups of the birth and death processes defined with the aid of the generalized binomial coefficients. The reason for doing so is revealed by the following theorem.

\begin{theorem}\label{thm:connection}
The restrictions of the semigroups $(R_t^\pm)_t$ to $\pi_{\pm1}\left(\chk\right)$ yield the semigroups of transition probabilities on $\Lambda$ given by \eqref{ptplus} and \eqref{ptminus}.
\end{theorem}

\begin{proof}
Following Biane \cite{biane1998}, we will find the formula for $(p_t^+)_t$ by decomposing a positive definite function $\varphi_{1,\alpha}\exp\left[-\left(e^t-1\right)\psi\right]$ into a convex combination
\begin{equation}\label{rozpiska1}
\varphi_{1,\alpha}\exp\left[-\left(e^t-1\right)\psi\right]=\sum_{\beta\in\Lambda} p_t^+(\alpha,\beta) \varphi_{\exp(t),\beta}.
\end{equation}
The left hand side of \eqref{rozpiska1} is equal to
\[
\exp\left[e^t i w-\frac{1}{2}e^t\norma{}{z}^2\right]q_\alpha(z).
\]
Hence, from Lemma \ref{lem:qcz} taken with $c=\exp(-t)$ and $z$ rescaled as $z\sqrt{\exp(t)}$, it equals
\begin{multline*}
\exp\left[e^{t}iw-\frac{1}{2}e^{t}\norma{}{z}^2]\right]\sum_{|\beta|\le|\alpha|}
\genbin{\alpha}{\beta}\left(e^{-t}\right)^{|\beta|}\left(1-e^{-t}\right)^{|\alpha|-|\beta|}q_\beta\left(z\sqrt{e^t}\right)\\
=\sum_{|\beta|\le|\alpha|}\genbin{\alpha}{\beta}\left(e^{-t}\right)^{|\beta|}\left(1-e^{-t}\right)^{|\alpha|-|\beta|}
\varphi_{\exp(t),\beta}.
\end{multline*}
Comparing this with the right hand side of \eqref{rozpiska1} we get the desired conclusion.

Similarly, for the semigroup $(p_t^-)_t$, the decomposition analogous to \eqref{rozpiska1} reads
\begin{equation}\label{rozpiska2}
\varphi_{-1,\alpha}\exp\left[-\left(1-e^{-t}\right)\psi\right]=\sum_{\beta\in\Lambda} p_t^-(\alpha,\beta) \varphi_{-\exp(-t),\beta}.
\end{equation}
Since the left hand side of \eqref{rozpiska2} is
\begin{equation*}
\exp\left[-\left(1-e^{-t}\right)\norma{}{z}^2\right]q_\alpha(z)\exp\left[-e^{-t}iw-\frac{1}{2}e^{-t}\norma{}{z}^2\right],
\end{equation*}
we can use Lemma \ref{lem:explag} with $c=1-\exp(-t)$ to expand it into
\begin{multline*}
\sum_{\beta\in\Lambda}\genbin{\beta}{\alpha}\frac{d_\beta}{d_\alpha}\left(1-e^{-t}\right)^{|\beta|-|\alpha|} \left(e^{-t}\right)^{|\alpha|+n}q_{\beta}\left(z\sqrt{e^{-t}}\right)\exp\left[-\left(1-e^{-t}\right)\norma{}{z}^2\right]=\\
\sum_{\beta\in\Lambda}\genbin{\beta}{\alpha}\frac{d_\beta}{d_\alpha}\left(1-e^{-t}\right)^{|\beta|-|\alpha|} \left(e^{-t}\right)^{|\alpha|+n}
\varphi_{-\exp(-t),\beta},
\end{multline*}
and the formula for $(p_t^-)_t$ follows from comparing the latter expression with the right hand side of \eqref{rozpiska2}.
\end{proof}

\begin{remark}\label{rem:Biane}
In \cite{biane1998}, as an example, Biane considered the quantum Ornstein-Uhlenbeck process associated with the Gelfand pair $(K,H)$, with $K=\textrm{SO}(\rzecz,n)\times T(1)$ and $H=\zesp^n\times\rzecz$. Biane suggested a way (different from our approach) to compute the restrictions of the semigroups $(R_t^\pm)_t$ to $\pi_{\pm1}\left(\chk\right)$ in this case. However, the required computations were quite tedious, so in \cite{biane1998} the transition rates were given only for $n=2$ and in the limiting case $n\to\infty$. Thus the rates obtained for action (iv) in Example \ref{ex:fiveactions} complement and enhance Biane's finding. 

It is worth noticing that although our derivations of the rates presented in Example \ref{ex:fiveactions} formally require $n\ge 3$, the formulas we get agree when $n=2$ with the ones obtained by Biane as well. 
It is important, however, to keep in mind that, due to a different indexing Biane uses for the decomposition \eqref{decompos}, his processes live on the pairs of non-negative integers $(d,m)$ (see p. 65 of \cite{biane1998}), which correspond to the partitions $(d+m,m)$ in our setup. 
\end{remark}

\section{Auxiliary results}\label{s:aux}
In this section we collect some technical results which have been used in the paper.

\begin{lemma}\label{lem:laguerre_gamma}
For every $k\in\nat$
\begin{equation}
\sum_{|\alpha|=k} d_\alpha q_\alpha(z)=L^{(n-1)}_k(\gamma(z)),
\end{equation}
where $L^{(n-1)}_k$ is the $k$-th (generalized) Laguerre polynomial of order $(n-1)$ given by \eqref{laguerre}.
\end{lemma}

\begin{proof}
Benson, Jenkins and Ratcliff in \cite[Proposition 4.7]{BJR92} proved the formula
\begin{equation}
p_\alpha\left(\frac{\partial}{\partial \bar{z}},-\frac{\partial}{\partial z}\right)\left(e^{-\gamma}\right)=q_\alpha e^{-\gamma},
\end{equation}
in which $p_\alpha\left(\frac{\partial}{\partial \bar{z}},-\frac{\partial}{\partial z}\right)$ stands for the differential operator on $V_\rzecz$ obtained by replacing each occurrence of $z_j$ in $p_\alpha$ by $\frac{\partial}{\partial \bar{z}_j}$, and each occurrence of $\bar{z}_j$ by $-\frac{\partial}{\partial z_j}$. The formula, along with Lemma 3.2 from \cite{BensonRatcliff98} asserting that
\begin{equation}\label{BRCombiLemma32}
\sum_{|\alpha|=m}d_\alpha p_\alpha=\frac{\gamma^m}{m!},
\end{equation}
implies 
\begin{equation*}
\sum_{|\alpha|=k}d_\alpha q_\alpha=
e^\gamma \sum_{|\alpha|=k}d_\alpha q_\alpha e^{-\gamma}=
e^\gamma\frac{\gamma^k\left(\frac{\partial}{\partial \bar{z}},-\frac{\partial}{\partial z}\right)\left(e^{-\gamma}\right)}{k!}=
e^\gamma\frac{(-1)^k}{ k!} \Delta^k\left(e^{-\gamma}\right),
\end{equation*}
where 
\[
\Delta=\sum_{j=1}^n \frac{\partial}{\partial z_j}\frac{\partial}{\partial \bar{z}_k}.
\]
Since
\[
\Delta^k\left(e^{-\gamma}\right)=(-1)^k k! e^{-\gamma} L_k^{(n-1)}(\gamma),
\]
the lemma follows.
\end{proof}

\begin{corollary}\label{cor:aux1}
For all $k\in\nat$ and $\beta\in\Lambda$ with $|\beta|\le k$,
\begin{equation*}
\sum_{|\alpha|=k}d_\alpha\genbinab=d_\beta\binom{k+n-1}{k-|\beta|}.
\end{equation*}
\end{corollary}

\begin{proof}
Comparing homogeneous parts of degree $2l$ ($l=0,\ldots,k$) on both sides of the equality
\[
\sum_{|\alpha|=k} d_\alpha\sum_{|\beta|\le|\alpha|}(-1)^\beta\genbinab p_\beta=\sum_{i=0}^k (-1)^i\binom{k+n-1}{k-i}\frac{\gamma^i}{i!}
\]
(which follows from Lemma \ref{lem:laguerre_gamma} and \eqref{qinp}), gives
\[
\sum_{|\alpha|=k,|\beta|=l} d_\alpha \genbinab p_\beta=\binom{k+n-1}{k-l}\frac{\gamma^l}{l!}.
\]
Hence, by \eqref{BRCombiLemma32},
\[
\sum_{|\alpha|=k,|\beta|=l} d_\alpha \genbinab p_\beta=\sum_{|\beta|=l}d_\beta\binom{k+n-1}{k-l} p_\beta,
\]
and the corollary follows immediately by comparing the coefficients of $p_\beta$ on both sides.
\end{proof}

\begin{corollary}\label{cor:sums_of_rates}
For $\beta\in\Lambda$
\begin{eqnarray*}
\sum_{|\alpha|=|\beta|+1} \frac{d_\alpha}{d_\beta}\genbinab&=&n+|\beta|,\\
\sum_{|\alpha|=|\beta|-1}\genbinba&=&|\beta|.
\end{eqnarray*}
\end{corollary}

\begin{proof}
The first equality is just a particular case of Corollary \ref{cor:aux1} with $k=|\beta|+1$. For the proof of the second equality, we use the following observation from \cite[Proposition 3.7]{BensonRatcliff98}:
\[
\gamma q_\beta=-\sum_{|\alpha|=|\beta|+1}\frac{d_\alpha}{d_\beta}\genbin{\alpha}{\beta}q_\alpha
+(2|\beta|+n)q_\beta-\sum_{|\alpha|=|\beta|-1}\genbin{\beta}{\alpha}q_\alpha.
\]
Evaluating both sides at $z=0$ and recalling the normalization of the polynomials $(q_\alpha)_\alpha$, we get
\[
\sum_{|\alpha|=|\beta|+1}\frac{d_\alpha}{d_\beta}\genbin{\alpha}{\beta}+\sum_{|\alpha|=|\beta|-1}\genbin{\beta}{\alpha}=
(2|\beta|+n).
\]
\end{proof}
\begin{lemma}\label{lem:sum}
For $m\in\nat$ and $\alpha\in\Lambda$ with $|\alpha|\ge m$, 
\begin{equation}\label{64}
\sum_{|\beta|=m}\genbin{\alpha}{\beta}=\binom{|\alpha|}{m}.
\end{equation}
\end{lemma}

\begin{proof}
We proceed by induction on $k=|\alpha|-m$. The result holds for $k=0$ by \eqref{zeroone}, and for $k=1$ by Corollary \ref{cor:sums_of_rates}. For the inductive step, let $k\ge1$ and assume \eqref{64} for $k-1$. By \eqref{BRCorr34}
\begin{equation*}
\sum_{|\lambda|=|\beta|+1}\genbin{\alpha}{\lambda}\genbin{\lambda}{\beta}=k\genbin{\alpha}{\beta}.
\end{equation*}
Hence
\begin{equation*}
\sum_{|\beta|=m}\genbin{\alpha}{\beta}=\frac{1}{k}\sum_{|\beta|=m}\sum_{|\lambda|=m+1}\genbin{\alpha}{\lambda}\genbin{\lambda}{\beta}.
\end{equation*}
Exchanging the order of summation and using Corollary \ref{cor:sums_of_rates} one gets
\begin{equation*}
\sum_{|\beta|=m}\genbin{\alpha}{\beta}=\frac{m+1}{k}\sum_{|\lambda|=m+1}\genbin{\alpha}{\lambda}=\frac{m+1}{k}\binom{|\alpha|}{m+1}=\binom{|\alpha|}{m}
\end{equation*}
(the second-to-last equality follows by using the inductive hypothesis), which completes the induction.
\end{proof}

\begin{lemma}\label{lem:explag}
For any $c\in(0,1)$ and $\lambda\in\Lambda$
\begin{equation}\label{explag}
\exp[-c\gamma(z)]q_\lambda(z)=\sum_{\beta\in\Lambda}\genbin{\beta}{\lambda}\frac{d_\beta}{d_\lambda}c^{|\beta|-|\lambda|}(1-c)^{|\lambda|+n}q_\beta(z\sqrt{1-c}),
\end{equation}
with the series converging absolutely and uniformly on compact subsets of $V$.
\end{lemma}

\begin{proof}
It was proved in \cite{yan_unpub} (see also \cite[Proposition 3.1]{BJR1998}) that the set $(q_\lambda)_{\lambda\in\Lambda}$ is a complete orthogonal system in the space $L^2_K(V_\rzecz,\exp(-\gamma(z)))$ of $K$-invariant functions that are square-integrable with respect to the weight appearing in the Fock inner product $\iloskal{\cdot}{\cdot}_1$, and that
\begin{equation}\label{iloskalqq}
\iloskal{q_\lambda}{q_\lambda}_1=\frac{1}{d_\lambda}.
\end{equation}

Let $\tilde{q}_\lambda(z)=q_\lambda(z\sqrt{1-c})$ and $g(z)=\exp[-c\gamma(z)]q_\lambda(z)$. Clearly, $(\tilde{q}_\lambda)_{\lambda\in\Lambda}$ is an orthogonal system in $L^2_K(V_\rzecz,\exp(-(1-c)\gamma(z)))$, the function $g$ belongs to $L^2_K(V_\rzecz,\exp(-(1-c)\gamma(z)))$, and $\iloskal{\tilde{q}_\lambda}{\tilde{q}_\lambda}_{1-c}=\iloskal{q_\lambda}{q_\lambda}_1=1/d_\lambda$. Therefore we can expand $g$ as
\[
\exp[-c\gamma(z)]q_\lambda(z)=\sum_{\beta\in\Lambda} a_\beta \tilde{q}_\beta(z),
\]
where $a_\beta=\iloskal{g}{\tilde{q}_\beta}_{1-c}/\iloskal{\tilde{q}_\beta}{\tilde{q}_\beta}_{1-c}$. By straightforward calculations,
\[
a_\beta=d_\beta(1-c^n)\iloskal{q_\lambda}{\tilde{q}_\beta}_1.
\]
Using \eqref{qinp} and the homogeneity of the polynomials $(p_\beta)_\beta$ to evaluate the latter inner product, we get
\begin{equation*}
\iloskal{q_\lambda}{\tilde{q}_\beta}_1=\sum_{|\gamma|\le|\beta|}\genbin{\beta}{\gamma}(-1)^{|\gamma|}(1-c)^{|\gamma|}\iloskal{q_\lambda}{p_\gamma}_1.
\end{equation*}
Formula (1.5) from \cite{BensonRatcliff98} stating that 
\begin{equation}\label{pinq}
p_\alpha=\sum_{|\beta|\le|\alpha|}(-1)^{|\beta|}\genbin{\alpha}{\beta}q_\beta,
\end{equation}
and \eqref{iloskalqq} together imply that
\begin{equation*}
\iloskal{q_\lambda}{p_\gamma}_1=
\begin{cases}
\frac{(-1)^{|\lambda|}}{d_\lambda}\genbin{\gamma}{\lambda}, & |\lambda|\le|\gamma|,\\
0, & |\lambda|>|\gamma|.
\end{cases}
\end{equation*}
Hence
\begin{equation*}
\iloskal{q_\lambda}{\tilde{q}_\beta}_1=\sum_{|\lambda|\le|\gamma|\le|\beta|}\genbin{\gamma}{\lambda}\genbin{\beta}{\gamma}
\frac{(-1)^{|\gamma|}(-1)^{|\lambda|}}{d_\lambda}(1-c)^{|\gamma|}=
\sum_{l=|\lambda|}^{|\beta|}\frac{(-1)^{l+|\lambda|}}{d_\lambda}(1-c)^l\sum_{|\gamma|=l}\genbin{\gamma}{\lambda}\genbin{\beta}{\gamma}.
\end{equation*}
By \eqref{BRCorr34},
\begin{equation*}
\iloskal{q_\lambda}{\tilde{q}_\beta}_1=\frac{1}{d_\lambda}\genbin{\beta}{\lambda}c^{|\beta|-|\lambda|}(1-c)^{|\lambda|},
\end{equation*}
and, consequently,
\[
a_\beta=\genbin{\beta}{\lambda}\frac{d_\beta}{d_\lambda}c^{|\beta|-|\lambda|}(1-c)^{|\lambda|+n}.
\]
Thus \eqref{explag} holds in $L^2_K(V_\rzecz,\exp(-(1-c)\gamma(z)))$. Observe that the series in \eqref{explag} also converges absolutely and uniformly on compact subsets of $V$. Indeed, since the spherical function $\varphi_{1-c,\lambda}$ is bounded by $1$ as an appropriately normalized matrix coefficient, we see that
\begin{multline*}
\sum_{\beta\in\Lambda} \left|\genbin{\beta}{\lambda}\frac{d_\beta}{d_\lambda}c^{|\beta|-|\lambda|}(1-c)^{|\lambda|+n} q_\beta(z\sqrt{1-c})\right|\le\\
\exp\left[\frac{(1-c)\gamma(z)}{2}\right](1-c)^{|\lambda|+n}\sum_\beta \genbin{\beta}{\lambda}\frac{d_\beta}{d_\lambda}c^{|\beta|-|\lambda|}=\\
\exp\left[\frac{(1-c)\gamma(z)}{2}\right](1-c)^{|\lambda|+n}\sum_{l=|\lambda|}^\infty c^{l-|\lambda|}\sum_{|\beta|=l}\genbin{\beta}{\lambda}\frac{d_\beta}{d_\lambda}.
\end{multline*} 
Using Corollary \ref{cor:aux1} and the equality
\[
\sum_{l=0}^\infty \binom{|\lambda|+n+l-1}{l}c^l(1-c)^{|\lambda|+n}=1,
\]
we arrive at the bound
\begin{equation*}
\sum_{\beta\in\Lambda} \left|\genbin{\beta}{\lambda}\frac{d_\beta}{d_\lambda}c^{|\beta|-|\lambda|}(1-c)^{|\lambda|+n} q_\beta(z\sqrt{1-c})\right|\le
\exp\left[\frac{(1-c)\gamma(z)}{2}\right],
\end{equation*}
which completes the proof of Lemma \ref{lem:explag}.
\end{proof}

\begin{lemma}\label{lem:qcz}
For $c\ge 0$,
\begin{equation*}
q_\alpha(\sqrt{c}z)=\sum_{|\beta|\le|\alpha|}\genbin{\alpha}{\beta} c^{|\beta|}(1-c)^{|\alpha|-|\beta|}q_\beta(z).
\end{equation*}

\end{lemma}

\begin{proof}
Formula \eqref{qinp}, the homogeneity of the polynomials $(p_\beta)_\beta$, and \eqref{pinq},
imply the following chain of equalities:
\begin{multline}\label{20180704}
q_\alpha(\sqrt{c}z)=\sum_{|\gamma|\le|\alpha|}(-1)^{|\gamma|}\genbin{\alpha}{\gamma} c^{|\gamma|} p_\gamma(z)=
\sum_{|\gamma|\le|\alpha|} (-1)^{|\gamma|}\genbin{\alpha}{\beta}c^{|\gamma|}
\sum_{|\beta|\le|\gamma|}(-1)^{|\beta|}\genbin{\gamma}{\beta}q_\beta(z)\\
=\sum_{|\beta|\le|\alpha|}(-1)^{|\beta|} q_{\beta}(z)\sum_{|\beta|\le|\gamma|\le|\alpha|} (-1)^{|\gamma|} c^{|\gamma|}\genbin{\alpha}{\gamma}\genbin{\gamma}{\beta}.
\end{multline}
Now we use \eqref{BRCorr34} to expand the last sum in \eqref{20180704} as
\begin{equation*}
\sum_{l=|\beta|}^{|\alpha|}(-1)^l c^l \sum_{|\gamma|=l}\genbin{\alpha}{\gamma}\genbin{\gamma}{\beta}=
\genbin{\alpha}{\beta}\sum_{l=|\beta|}^{|\alpha|}(-c)^l \binom{|\alpha|-|\beta|}{|\alpha|-l}=
\genbin{\alpha}{\beta}(-c)^{|\beta|}(1-c)^{|\alpha|-|\beta|},
\end{equation*}
which proves the lemma.
\end{proof}

\subsection*{Acknowledgment} We thank Philippe Biane for guiding us patiently through intricacies of some of his papers. We also thank W{\l}odek Bryc and Jacek Weso{\l}owski for insightful comments.

\bibliographystyle{amsplain}

\def\cprime{$'$}
\providecommand{\bysame}{\leavevmode\hbox to3em{\hrulefill}\thinspace}
\providecommand{\MR}{\relax\ifhmode\unskip\space\fi MR }
\providecommand{\MRhref}[2]{%
	\href{http://www.ams.org/mathscinet-getitem?mr=#1}{#2}
}
\providecommand{\href}[2]{#2}

\end{document}